\documentclass[12pt]{article}
\usepackage{amsmath,amssymb,amsthm}
\newtheorem{theorem}{Theorem}[section]
\newtheorem{corollary}[theorem]{Corollary}
\newtheorem{lemma}[theorem]{Lemma}
\newtheorem{proposition}[theorem]{Proposition}
\newtheorem{definition}[theorem]{Definition}
\newtheorem{remark}[theorem]{Remark}

\numberwithin{equation}{section}
\def\L{\mathcal{L}}

\def\A{\mathcal{A}}

\def\B{\mathcal{B}}
\def\E{\mathcal{E}}

\begin{document}
\title{An improvement on the number of simplices in $\mathbb{F}_q^d$}
\author{
Pham Duc Hiep
\and 
    Thang Pham
  \and
    Le Anh Vinh}
\date{}
\maketitle

\begin{abstract}
Let $\E$ be a set of points in $\mathbb{F}_q^d$. Bennett, Hart, Iosevich, Pakianathan, and Rudnev  (2016) proved that if $|\E|\gg q^{d-\frac{d-1}{k+1}}$ then $\E$ determines a positive proportion of all $k$-simplices.  In this paper, we give an improvement of this result in the case when $\E$ is the Cartesian product of sets. Namely, we show that if $\E$ is the Cartesian product of sets and $q^{\frac{kd}{k+1-1/d}}=o(|\E|)$, the number of congruence classes of $k$-simplices determined by $\E$ is at least $(1-o(1))q^{\binom{k+1}{2}}$, and  in some cases our result is sharp.
\end{abstract}

\section{Introduction}
Let $\mathbb{F}_q$ be a finite field of order $q$ with $q=p^r$ for some prime $p$ and positive integer $r$. Denote by $O(d, \mathbb{F}_q)$ the orthogonal group in $\mathbb{F}_q^d$. We say that two $k$-simplices  in $\mathbb{F}_q^d$ with vertices $(\mathbf{x}_1, \ldots, \mathbf{x}_{k+1})$, $(\mathbf{y}_1, \ldots, \mathbf{y}_{k+1})$ are in the same congruence class if there exist $\theta\in O(d, \mathbb{F}_q)$ and $\mathbf{z}\in \mathbb{F}_q^d$ so that $\mathbf{z}+\theta(\mathbf{x}_i)=\mathbf{y}_i$ for all $i=1, 2, \ldots, k+1$.

Hart and Iosevich \cite{hart1} made the first investigation on counting the number of congruence classes of simplices determined by a point set in $\mathbb{F}_q^d$. More precisely, they proved that if $|\E|\gg q^{\frac{kd}{k+1}+\frac{k}{2}}$ with $d\ge \binom{k+1}{2}$, then $\E$ contains a copy of all $k$-simplices with non-zero edges. Here and throughout,  $X \ll Y$ means that there exists $C>0$ such that $X\le  CY$, and $X=o(Y)$ means that $X/Y\to 0$ as $q\to \infty$, where $X, Y$ are viewed as functions in $q$.

Using methods from spectral graph theory, the third listed author \cite{vinh3} improved this result. In particular, he showed that the same result also holds when $d\ge 2k$ and $|\E|\gg q^{(d-1)/2+k}$. It follows from the results in \cite{hart1, vinh3} that the most difficulties arise when the size of simplex is large with respect to the dimension of the space, for instance, the result in \cite{vinh3} on the number of congruence classes of triangles is only non-trivial if $d\ge 4$. 

In \cite{cover}, Covert et al. addressed the case of triangles in $\mathbb{F}_q^2$, and they established that if $|\E|\gg \rho q^2$, then $\E$ determines at least $c\rho q^3$ congruence classes of triangles. The author of \cite{vinh2} extended this result to the case $d\ge 3$. Formally, he proved that if $|\E|\gg q^{\frac{d+2}{2}}$, then $\E$ determines a positive proportion of all triangles. Using Fourier analytic techniques, Chapman et al. \cite{chapman} indicated that the threshold $q^{\frac{d+2}{2}}$ on the cardinality of $\E$ in the triangle case can be replaced by $q^{\frac{d+k}{2}}$ for the case of $k$-simplices. In a recent result, Bennett et al. \cite{groupaction} improved the threshold $q^{\frac{d+k}2}$ to $q^{d-\frac{d-1}{k+1}}$. The precise statement is given by the following theorem. 
\begin{theorem}[\cite{groupaction}]\label{thm1}
Let $\E$ be a subset in $\mathbb{F}_q^d$. Suppose that 
\[|\E|\gg q^{d-\frac{d-1}{k+1}},\]
then, for $1\le k\le d$, the number of congruence classes of $k$-simplices determined by $\E$ is at least $cq^{\binom{k+1}{2}}$ for some positive constant $c$. 
\end{theorem}
In this paper, by using methods from spectral graph theory and elementary results on group actions, we improve Theorem \ref{thm1} in the case $\E$ has Cartesian product structure. In particular, we have the following result.
\begin{theorem}\label{main}
Let $\E=\A_1\times \cdots\times A_d$ be a subset in $\mathbb{F}_q^d$. Suppose that 
\[(\min_{1\le i\le d}|\A_i|)^{-1}|\E|^{k+1}\gg q^{kd},\] 
then for $1\le k\le d$, the number of congruence classes of $k$-simplices determined by $\E$ is at least $cq^{\binom{k+1}{2}}$ for some positive constant $c$. 
\end{theorem}
\begin{corollary}\label{htv}
Let $\E=\A^d$ be a subset in $\mathbb{F}_q^d$. If $|\E|\gg q^{\frac{kd}{k+1-1/d}}$ then the number of congruence classes of simplices determined by $\E$ is at least $cq^{\binom{k+1}{2}}$ for some positive constant $c$.
\end{corollary}
As a consequence of Corollary \ref{htv}, we recover the following result in \cite{hieu}.
\begin{theorem}
Let $\mathcal{A}$ be a subset in $\mathbb{F}_q$. If $|\mathcal{A}|\gg q^{\frac{d^2}{2d-1}}$, then the number of distinct distances determined by points in $\mathcal{A}^d\subseteq \mathbb{F}_q^d$ is at least $\gg q$.
\end{theorem}
\paragraph{On the number of congruence classes of triangles in $\mathbb{F}_q^2$.}
For the case of triangles in $\mathbb{F}_q^2$, in 2012 Bennett, Iosevich, and Pakianathan \cite{be}, using Elekes-Sharir paradigm and an estimate on the number of incidences between points and lines in $\mathbb{F}_q^3$, improved significantly the result in \cite{cover}. In particular, they proved that if $|\E|\gg q^{7/4}$ and $q\equiv 3\mod 4$, then the number of triangles determined by $\E$ is at least $cq^3$ for some positive constant $c$.  The authors of \cite{groupaction} recently improved the exponent $7/4$ to $8/5$ in the following.
\begin{theorem}[Bennett et al. \cite{groupaction}]\label{tamgiac}
Let $\E$ be a set of points in $\mathbb{F}_q^2$. If $|\E|\gg q^{8/5}$, then $\E$ determines a positive proportion of all triangles.
\end{theorem}
We will give a graph-theoretic proof for this theorem in Section $4$. If $\E$ has Cartesian product structure of sets with different sizes, as a consequence of Theorem \ref{main}, we are able to obtain a much stronger result as follows.
\begin{theorem}\label{tot}
Let $\A, \B$ be subsets in $\mathbb{F}_q$. If $|\A|\ge q^{\frac{1}{2}+\epsilon}$ and $|\B|\ge q^{1-\frac{2\epsilon}{3}}$ for some $\epsilon\ge 0$, then the number of congruence classes of triangles determined by $\A\times \B\subseteq\mathbb{F}_q^2$ is at least $cq^3$ for some positive constant $c$. 
\end{theorem}
Note that if $\A$ and $\B$ are arbitrary sets in $\mathbb{F}_q$, then it follows from Theorem \ref{main} that in order to prove that there exist at least $cq^3$ congruence classes of triangles, we need the condition $|\A|^2|\B|^3\gg q^4$. In particular, if $|\A|<q^{1/2}$ then we must have $|\B|> q$. In fact,  one can not expect to get a positive proportion of congruence of triangles in the set $\A\times \B$ with arbitrary sets $\A$ and $\B$ satisfying  $|\A|=o(q^{1/2})$ and $|\B|<q$, since the authors of \cite{groupaction} gave a construction with $|\A|=q^{1/2-\epsilon'}$ and $|\B|=q$, and the number of congruence classes triangles determined by $\E$ is at most  $cq^{3-\epsilon"}$ for $\epsilon">0$. Therefore the result in the form of Theorem \ref{tot} is tight up to a factor of $q^{\epsilon/3}$.
\paragraph{Distinct distance subsets in $\mathbb{F}_q^2$.}
Given a subset $\E\subset \mathbb{F}_q^d$,  a subset $U\subset \E$ is called a \textit{distinct distance subset} if there are no four distinct points $\mathbf{x}, \mathbf{y}, \mathbf{z}, \mathbf{t}\in U$ such that $||\mathbf{x}-\mathbf{y}||=||\mathbf{z}-\mathbf{t}||$. In \cite{phuong}, Phuong et al. studied the finite field analogue of this problem. More precisely, the authors of \cite{phuong} proved that if $|\E|\ge 2q^{(2d+1)/3}$, then there exists a distinct distance subset of cardinality $\gg q^{1/3}$.  This implies that the result is only non-trivial when $d\ge 3$. In this paper, we fill in this gap. In particular, we prove that if $d=2$, then the threshold $q^{(2d+1)/3}$ can be improved to $q^{4/3}$. In particular, the statement is in the following.
\begin{theorem}\label{khoangcachphanbiet}
Let $\E$ be a subset in $\mathbb{F}_q^d $ with $|\E|\gg q^{4/3}.$ There exists a distinct distance subset $U\subseteq \E$ satisfying $ q^{1/3} \ll |U|\ll q^{1/2}$. 
\end{theorem}

\section{Tools from spectral graph theory}
 A graph $G= (V, E)$ is called an $(n, d, \lambda)$-graph if it is $d$-regular, has $n$
vertices, and the second eigenvalue of $G$ is at most $\lambda$. The result below gives an estimate on the number of edges between two multi-sets of vertices in an $(n, d, \lambda)$-graph.
\begin{lemma}[\cite{hanson}]\label{edge}
  Let $G = (V, E)$ be an $(n, d, \lambda)$-graph. For any two multi-sets of vertices $B, C$, we have
  \[ \left| e (B, C) - \frac{d|B | |C|}{n} \right| \leq \lambda \left(\sum_{b\in B}m_B(b)^2\right)^{1/2}\left(\sum_{c\in C}m_C(c)^2\right)^{1/2}, \]
  where $e(B,C)$ is the number of edges between $B$ and $C$ in $G$, and $m_X(x)$ is the multiplicity of $x$ in X.
\end{lemma}
Let $PG (q, d)$ denote the projective geometry of dimension $d-1$ over finite field $\mathbb{F}_q$. The vertices of $P G (q, d)$
correspond to the equivalence classes of the set of all non-zero vectors $[\mathbf{x}] = [x_1, \ldots, x_d]$ over $\mathbb{F}_q$, where two vectors are equivalent if
one is a multiple of the other by a non-zero element of the field. 

In this section, we recall a well-known construction of the Erd\H{o}s-R\'{e}nyi graph due to Alon and Krivelevich \cite{alon} as follows. Let $\mathcal{ER}(\mathbb{F}_q^d)$
denote the graph whose vertices are the points of $P G (q, d)$, and two (not
necessarily distinct) vertices $[\mathbf{x}]=[x_1, \ldots, x_d]$ and $[\mathbf{y}]=[y_1, \ldots, y_d]$ are adjacent if and only if $x_1 y_1+ \ldots + x_d y_d = 0$.  Alon and Krivelevich \cite{alon} obtained the following result on the spectrum of $\mathcal{ER}(\mathbb{F}_q^d)$.
\begin{lemma}[Alon and Krivelevich, \cite{alon}]\label{erdosrenyi}
For any odd prime power $q$  and $d\geq 2$, the Erd\H{o}s-R\'enyi graph $\mathcal{ER}(\mathbb{F}_q^d)$ is an 
\[ \left( \frac{q^d-1}{q-1}, \frac{q^{d-1}-1}{q-1},q^{(d-2)/2} \right) - \mbox{graph}.\]
\end{lemma}
The next lemma is useful in the proof of Theorem \ref{main}, which allows us to reduce $k$-simplices to $2$-simplices. 
\begin{lemma}[Bennett et al., \cite{groupaction}]\label{bodebennet}
Let $V$ be a finite space and $f\colon V\to \mathbb{R}_{\ge 0}$ a function. For any $n\ge 2$ we have
\[\sum_{z\in V}f^n(z)\le |V|\left( \frac{||f||_1}{|V|}\right)^n+\frac{n(n-1)}{2}||f||_{\infty}^{n-2}\sum_{z\in V}\left(f(z)-\frac{||f||_1}{|V|}\right)^2, \]
where $||f||_1=\sum_{z\in V}|f(z)|,$ and $||f||_\infty=\max_{z\in V}f(z)$.
\end{lemma}
\section{Proof of Theorem \ref{main}}\label{hoho}
For $\mathcal{E}=\mathcal{A}_1\times \cdots\times \mathcal{A}_d\subseteq \mathbb{F}_q^d$ and $t\in \mathbb{F}_q$, we define
$\nu_{\mathcal{E}}(t)$ as the cardinality of the set $\left\lbrace (\mathbf{x}, \mathbf{y})\in \mathcal{E}\times \mathcal{E}\colon ||\mathbf{x}-\mathbf{y}||=t\right\rbrace.$
In order to prove Theorem \ref{main} we first need the following lemma.
\begin{lemma}\label{mainlm}
For $\mathcal{E}=\mathcal{A}_1\times \cdots\times \mathcal{A}_d\subseteq \mathbb{F}_q^d$ with $|\A_d|=\min_{1\le i\le d}|\A_i|$, 
\[\sum_{t\in \mathbb{F}_q}\nu_{\mathcal{E}}(t)^2< \frac{|\mathcal{E}|^{4}}{q}+2q^{d-1}|\mathcal{E}|^{2}|\mathcal{A}_d|.\]
\end{lemma}
\begin{proof}
For a fixed pair $(a, b)\in \A_d^2$, let $N(a, b)$ denote the set  of quadruples $(\mathbf{x}, \mathbf{y}, \mathbf{z}, \mathbf{t})\in \mathcal{E}^4$ with $\mathbf{x}=(x_1, \ldots, x_{d-1}, a)$ and $\mathbf{y}=(y_1, \ldots, y_{d-1}, b)$ satisfying $||\mathbf{x}-\mathbf{z}||=||\mathbf{y}-\mathbf{t}||$. Then one has 
\[\sum_{t\in \mathbb{F}_q}\nu_{\mathcal{E}}(t)^2=\sum_{(a, b)\in \A_d^2}|N(a,b)|\le |\A_d|^2 \max_{(a, b)\in \A_d^2}|N(a,b)|.\]
We next show that 
\[|N(a,b)|\le \frac{|\mathcal{E}|^{2}|\A_d|^{-2}}{q}+q^{d-1}|\mathcal{E}|^{2}|\A_d|^{-1}, ~\forall (a, b)\in \A_d^2.\]
Indeed, let $U$ and $V$ be, respectively, multi-subsets in $PG(q, 2d)$ defined by
\[U=\left\lbrace \left[-2x_1, \ldots, -2x_{d-1}, 1, t_1, \ldots, t_{d-1}, -(t_d-b)^2-\sum_{i=1}^{d-1}t_i^2+\sum_{i=1}^{d-1}x_i^2\right]\colon x_i, ~t_i\in \A_i\right\rbrace,\]
and 
\[V=\left\lbrace \left[z_1, \ldots, z_{d-1}, (z_d-a)^2+\sum_{i=1}^{d-1}z_i^2-\sum_{i=1}^{d-1}y_i^2, 2y_1, \ldots, 2y_{d-1}, 1\right]\colon z_i, ~t_i\in \A_i\right\rbrace.\]
It is clear that 
$$|U|= |\mathcal{E}||\A_d|^{-1}, |V|= |\mathcal{E}||\A_d|^{-1}, ~m_{U}(\mathbf{u})\le 2, ~m_{V}(\mathbf{v})\le 2, \forall \mathbf{u}\in U, \mathbf{v}\in V.$$ 
and $|N(a,b)|$ is equal to the number of edges between $U$ and $V$ in the Erd\H{o}s-R\'{e}nyi graph $\mathcal{ER}(\mathbb{F}_q^{2d})$. Thus it follows from Lemmas \ref{edge} and  \ref{erdosrenyi} that 
\[|N(a,b)|<\frac{|\mathcal{E}|^{4}|\A_d|^{-2}}{q}+2q^{d-1}|\mathcal{E}|^2|\A_d|^{-1},\]
and this completes the proof of the lemma.
\end{proof}
It is convenient to recall the following definition which is given in \cite{groupaction}.
\begin{definition}
Let $V$ be the $\mathbb{F}_q$-vector space of $(k+1)\times (k+1)$ symmetric matrices $\mathbb {D}$ which can be viewed as the space of possible ordered $k$-simplex distances. For $\E\subset\mathbb{F}_q^d$, we define $\mu\colon V\to \mathbb{Z}$
\[\mu(\mathbb{D}):=\#\left\lbrace (\mathbf{x}_1, \ldots, \mathbf{x}_{k+1})\in \E^{k+1}\colon ||\mathbf{x}_i-\mathbf{x}_j||=d_{i,j}, 1\le i<j\le k+1\right\rbrace\]
\end{definition}
\begin{proof}[{\bf Proof of Theorem \ref{main}}]
Suppose that $|\A_d|=\min_{1\le i\le d}|\A_i|$. Let $T_{k,d}(\E)$ denote the set of congruence classes of $k$-simplices  determined by $\E$. It follows from the Cauchy-Schwarz inequality that
\[\sum_{\mathbb{D}}\mu(\mathbb{D})\le \left(\sum_{\mathbb{D}\in supp(\mu)}1\right)^{1/2}\left(\sum_{\mathbb{D}\in supp(\mu)}\mu(\mathbb{D})^2\right)^{1/2}=|T_{k, d}(\E)|^{1/2}\left(\sum_{\mathbb{D}\in supp(\mu)}\mu(\mathbb{D})^2\right)^{1/2}.\]
This gives 
\[|T_{k,d}(\E)|\ge \frac{(\sum_{\mathbb{D}}\mu(\mathbb{D}))^2}{\sum_{\mathbb{D}}\mu(\mathbb{D})^2}\ge\frac{|\E|^{2k+2}}{\sum_{\mathbb{D}}\mu(\mathbb{D})^2}.\]
For $\theta\in O(d, \mathbb{F}_q)$ and $\mathbf{z}\in \mathbb{F}_q^d$, we define 
\[w_\theta(\mathbf{z}):=\left\lbrace (\mathbf{u}, \mathbf{v})\in \E^2\colon \theta(\mathbf{u})+\mathbf{z}=\mathbf{v}\right\rbrace.\]
and denote the common stabilizer size of $k$-simplices in the congruence class $\mathbb{D}$ by $s(\mathbb{D})$. 
It has been shown in \cite{groupaction} that $s(\mathbb{D})\le |O(d-k, \mathbb{F}_q)|$, and $|O(n, \mathbb{F}_q)|=2(1+o(1))q^{\binom{n}{2}}$. Furthermore, it is easy to check that
\begin{equation}\label{hai}\sum_{\mathbb{D}}s(\mathbb{D})\mu(\mathbb{D})^2\le \sum_{\theta\in O(d, \mathbb{F}_q), \mathbf{z}\in \mathbb{F}_q^d}|w_\theta(\mathbf{z})|^{k+1},\end{equation}
where $|w_\theta(\mathbf{z})|$ is the cardinality of $w_\theta(\mathbf{z})$.

For a fixed $\theta$, it follows from Lemma \ref{bodebennet} with $f(z):=|w_\theta(z)|$, $||f||_1=|\E|^2$, and $||f||_\infty\le |\E|$ that 
\[
\sum_{\mathbf{z}\in \mathbb{F}_q^{d}}|w_\theta(\mathbf{z})|^{k+1}\le\frac{|\E|^{2k+2}}{q^{kd}}+\frac{k(k-1)}{2}|\E|^{k-1}\sum_{ \mathbf{z}\in \mathbb{F}_q^d}\left(|w_\theta(\mathbf{z})|-\frac{|\E|^2}{q^d}\right)^2.\]
Thus we obtain
\begin{align}\label{mot*}
\sum_{\theta\in O(d, \mathbb{F}_q), \mathbf{z}\in \mathbb{F}_q^{d}}|w_\theta(\mathbf{z})|^{k+1}&\le |O(d, \mathbb{F}_q)|\frac{|\E|^{2k+2}}{q^{kd}}+\frac{k(k-1)}{2}|\E|^{k-1}\sum_{ \theta, \mathbf{z}}\left(|w_\theta(\mathbf{z})|-\frac{|\E|^2}{q^d}\right)^2\\
&\le |O(d, \mathbb{F}_q)|\frac{|\E|^{2k+2}}{q^{kd}}+\frac{k(k-1)}{2}|\E|^{k-1}\left(\sum_{ \theta, \mathbf{z}}|w_\theta(\mathbf{z})|^2-\frac{|\E|^4|O(d, \mathbb{F}_q)|}{q^d}\right).\nonumber
\end{align}

It follows from the definition of $w_\theta(\mathbf{z})$ that $|w_\theta(\mathbf{z})|^2$ is equal to the number of quadruples $(\mathbf{a}, \mathbf{b}, \mathbf{c}, \mathbf{d})\in \E^4$ satisfying $\theta(\mathbf{a})+\mathbf{z}=\mathbf{c}$ and $\theta(\mathbf{b})+\mathbf{z}=\mathbf{d}$. This implies that $\theta(\mathbf{a}-\mathbf{b})=(\mathbf{c}-\mathbf{d})$, and $||\mathbf{a}-\mathbf{b}||=||\mathbf{c}-\mathbf{d}||$. Since the stabilizer of a non-zero element in $\mathbb{F}_q^d$ is at most $|O(d-1, \mathbb{F}_q)|$, it follows that each  quadruple $(\mathbf{a}, \mathbf{b}, \mathbf{c}, \mathbf{d})\in \E^4$, which satisfies $\mathbf{a}-\mathbf{b}\ne \mathbf{0}$, $||\mathbf{a}-\mathbf{b}||=||\mathbf{c}-\mathbf{d}||$, and $\theta(\mathbf{a}-\mathbf{b})=(\mathbf{c}-\mathbf{d})$ for some $\theta$, will be counted at most $|O(d-1, \mathbb{F}_q)|$ times in the sum $\sum_{ \theta, \mathbf{z}}|w_\theta(\mathbf{z})|^2$. If $\mathbf{a}=\mathbf{b}$ and $\mathbf{c}=\mathbf{d}$, then the quadruples $(\mathbf{a}, \mathbf{b}, \mathbf{c}, \mathbf{d})$ will be counted at most $|O(d, \mathbb{F}_q)|$ times in the sum $\sum_{ \theta, \mathbf{z}}|w_\theta(\mathbf{z})|^2$.

Let 
\[W:=\left\lbrace (\mathbf{a}, \mathbf{b}, \mathbf{c}, \mathbf{d})\in \E^4\colon ||\mathbf{a}-\mathbf{b}||=||\mathbf{c}-\mathbf{d}||\right\rbrace.\]

It is clear that $\sum_{t\in \mathbb{F}_q}\nu_\E(t)^2=|W|$. If $(\mathbf{a}, \mathbf{b})$ and $(\mathbf{c}, \mathbf{d})$ belong to $w_\theta(z)$ for some $\theta\in O(d, \mathbb{F}_q)$ and $\mathbf{z}\in \mathbb{F}_q^d$, then  $(\mathbf{a}, \mathbf{b}, \mathbf{c}, \mathbf{d})\in W$ and $(\mathbf{b}, \mathbf{a}, \mathbf{d}, \mathbf{c})\in W$. From this observation,  we get the following 
\begin{align}\label{hai*}\sum_{ \theta, \mathbf{z}}|w_\theta(\mathbf{z})|^2&\le \frac{|O(d-1, \mathbb{F}_q)||W|}{2}+|O(d, \mathbb{F}_q)||\E|^2\\
&\le \frac{|O(d-1, \mathbb{F}_q)|}{2}\left(\sum_{t\in \mathbb{F}_q}\nu_{\mathcal{E}}(t)^2\right)+|O(d, \mathbb{F}_q)||\E|^2,
\end{align}
where the factor $|\E|^2$ comes from the number of quadruples $(\mathbf{a}, \mathbf{b}, \mathbf{c}, \mathbf{d})$ with  $\mathbf{a}=\mathbf{b}$ and $\mathbf{c}=\mathbf{d}$.

Lemma \ref{mainlm} together with the inequalities 
$$|O(d-1, \mathbb{F}_q)|\dfrac{|\E|^4}{2q}\le \dfrac{|O(d, \mathbb{F}_q)||\E|^4}{q^d}$$ 
and 
$$|O(d, \mathbb{F}_q)||\E|^2\le |O(d-1, \mathbb{F}_q)|q^{d-1}|\E|^{2}|\A_d|$$ leads to
\begin{align*}
\sum_{ \theta, \mathbf{z}}|w_\theta(\mathbf{z})|^2&\le \frac{|O(d-1, \mathbb{F}_q)|}{2}\left(\frac{|\E|^4}{q}+2q^{d-1}|\E|^{2}|\A_d|\right)+|O(d, \mathbb{F}_q)||\E|^2\\
&\le 4|O(d-1, \mathbb{F}_q)|q^{d-1}|\E|^{2}|\A_d|,
\end{align*}

Combining (\ref{mot*}) with  (\ref{hai*}), we get 
\begin{align}\label{bon}
\sum_{\theta\in O(d, \mathbb{F}_q), \mathbf{z}\in \mathbb{F}_q^{d}}|w_\theta(\mathbf{z})|^{k+1}&\le |O(d, \mathbb{F}_q)|\frac{|\E|^{2k+2}}{q^{kd}}+2k(k-1)|\E|^{k-1}|O(d-1, \mathbb{F}_q)|(q^{d-1}|\E|^{2}|\A_d|)\nonumber\\
&\le |O(d, \mathbb{F}_q)|\frac{|\E|^{2k+2}}{q^{kd}}+2k(k-1)q^{d-1}|\E|^{k+1}|\A_d||O(d-1, \mathbb{F}_q)|.
\end{align}

It follows from (\ref{hai}) and (\ref{bon}) that
\[\sum_{\mathbb{D}}s(\mathbb{D})\mu(\mathbb{D})^2\le  |O(d, \mathbb{F}_q)|\frac{|\E|^{2k+2}}{q^{kd}}+2k(k-1)q^{d-1}|\E|^{k+1}|\A_d||O(d-1, \mathbb{F}_q)|.\]

Furthermore we have  $s(\mathbb{D})\le |O(d-k, \mathbb{F}_q)|$, this implies that 
\[\sum_{\mathbb{D}}\mu(\mathbb{D})^2\le \frac{|\E|^{2k+2}}{q^{\binom{k+1}{2}}}+2k(k-1)q^{kd-\binom{k+1}{2}}|\E|^{k+1}|\A_d|=(1+o(1)) \frac{|\E|^{2k+2}}{q^{\binom{k+1}{2}}}\]
when $q^{kd}=o\left(|\E|^{k+1}|\A_d|^{-1}\right)$. In other words, if $q^{kd}=o\left(|\E|^{k+1}|\A_d|^{-1}\right)$, then the number of congruence simplices determined by $\E$ satisfies
\[|T_{k,d}(\E)|=(1-o(1)) q^{\binom{k+1}{2}},\]
which ends the proof of the theorem.
\end{proof}
\section{Proof of Theorem \ref{tamgiac}}
First we need the following proposition.
\begin{proposition}\label{pro2}
Let $\mathcal{E}$ be a subset in $\mathbb{F}_q^2$ and $\nu_{\mathcal E}(\lambda)$ be the number of pairs $(\mathbf{p}, \mathbf{q})\in \E\times \E$ such that $||\mathbf{p}-\mathbf{q}||=\lambda$. If $|\mathcal{E}|\ge  4q$, then
\[\sum_{\lambda\in \mathbb{F}_q^*}\nu_{\mathcal{E}}(\lambda)^2\ll \frac{|\mathcal{E}|^4}{q}+q|\mathcal{E}|^{5/2}.\]
\end{proposition}
The proof of this proposition is based on the following lemma.

\begin{lemma}\label{1172016}
Let $\mathcal{E}$ be a subset in $\mathbb{F}_q^2$. For a fixed $\lambda\in \mathbb{F}_q^*$,  denote by $H_\lambda(\mathcal{E})$ the number of hinges of the form $(\mathbf{p}, \mathbf{q_1}, \mathbf{q_2})\in \mathcal{E}\times \mathcal{E}\times \mathcal{E}$ with $||\mathbf{p}-\mathbf{q_1}||=||\mathbf{p}-\mathbf{q_2}||=\lambda$. 
If $|\mathcal{E}|\ge  4q$, then
\[\sum_{\lambda\in \mathbb{F}_q^*}\nu_{\mathcal{E}}(\lambda)^2\le\frac{|\mathcal{E}|}{4}\sum_{\lambda\in \mathbb{F}_q^*} H_{\lambda}(\mathcal{E}).\]
\end{lemma}
\begin{proof}
By assumption $|\mathcal{E}|\ge 4q$ and note that there are at most $2q$ points on isotropic lines in the case $q\equiv 3\mod 4$, one may assume that there are no two points $(\mathbf{p}, \mathbf{q})$ in $\mathcal{E}\times \mathcal{E}$ with $||\mathbf{p}-\mathbf{q}||=0$. We now prove that 
\begin{equation}\label{lund11}
\sum_{\lambda\in \mathbb{F}_q^*}\nu_{\mathcal{E}}(\lambda)^2\le |\mathcal{E}|\sum_{\lambda\in \mathbb{F}_q^*}H_\lambda(\mathcal{E}).\end{equation}
For each point $\mathbf{p}$ in $\mathcal{E}$, let $x_{\mathbf{p}}^\lambda$ be the number of points $\mathbf{q}\in \mathcal{E}$ satisfying $||\mathbf{p}-\mathbf{q}||=\lambda$. Then one has 
\[H_{\lambda}(\mathcal{E})=\sum_{\mathbf{p}\in \mathcal{E}}(x_{\mathbf{p}}^\lambda)^2.\]
On the other hand, by applying the Cauchy-Schwarz inequality,  we obtain
\[\sum_{\lambda\in \mathbb{F}_q^*}\nu_{\mathcal{E}}(\lambda)^2=\frac{1}{4}\sum_{\lambda\in \mathbb{F}_q^*}\left(\sum_{\mathbf{p}\in \mathcal{E}}x_{\mathbf{p}}^\lambda\right)^2\le \frac{1}{4}\sum_{\lambda\in \mathbb{F}_q^*}|\mathcal{E}|H_{\lambda}(\mathcal{E})=\frac{1}{4}|\mathcal{E}|\sum_{\lambda\ne 0} H_{\lambda}(\mathcal{E}),\]
which completes the proof of the lemma.
\end{proof}

A reflection about a point $\mathbf{u}\in \mathbb{F}_q^2$ is a map of the form
\[R_\mathbf{u}(\mathbf{x})=R(\mathbf{x}-\mathbf{u})+\mathbf{u},\]
where $R$ is a matrix of the form 
\[R=\begin{pmatrix}
a&b\\
b&-a
\end{pmatrix}, \quad a, b\in \mathbb{F}_q, ~a^2+b^2=1.\]

For $\lambda\in \mathbb{F}_q^*$,  the \textit{reflection graph} $RF_\lambda(\mathbb{F}_q^2)$ is constructed as
 \[V(RF_\lambda(\mathbb{F}_q^2))=\left\lbrace (\mathbf{x}, \mathbf{y})\in \mathbb{F}_q^2\times \mathbb{F}_q^2\colon ||\mathbf{x}-\mathbf{y}||=\lambda\right\rbrace,\]
and 
\[E(RF_\lambda(\mathbb{F}_q^2))=\left\lbrace\left((\mathbf{x}, \mathbf{y}), (\mathbf{z}, \mathbf{w})\right)\in V(RF_\lambda(\mathbb{F}_q^2))\times V(RF_\lambda(\mathbb{F}_q^2))\colon \exists R_\mathbf{u}, ~ R_{\mathbf{u}}(\mathbf{x})=\mathbf{z}, R_{\mathbf{u}}(\mathbf{y})=\mathbf{w}\right\rbrace.\]
Hanson et al. \cite{hanson} established the $(n, d, \lambda)$ form of this graph as follows.
\begin{lemma}[Hanson et al., \cite{hanson}]\label{reflection}
The reflection graph $RF_\lambda(\mathbb{F}_q^2)$ is 
\[\left(q^2(q\pm 1), q^2\pm q, 2(q\pm1)\right) - \mbox{graph}.\]
\end{lemma}

The proof of Proposition \ref{pro2} below is quite similar to that of \cite[Theorem 1 and Lemma 14]{hanson}, but here we give a straight proof and avoid using the bound $\sum_{\lambda\in \mathbb{F}_q^*}\nu_{\mathcal{E}}(\lambda)^2\ll \frac{|\mathcal{E}|^4}{q}+q^2|\mathcal{E}|^2$ which can be proved by employing the distribution of edges between two sets of vertices in the Erd\H{o}s-R\'{e}nyi graph with an appropriated setting.

\begin{proof}[{\bf Proof of Proposition \ref{pro2}}]
We may again assume that there are no two distinct points $\mathbf{p}$ and $\mathbf{q}$ in $\mathcal{E}$ satisfying $||\mathbf{p}-\mathbf{q}||=0$ (as in the proof of Lemma \ref{1172016}).

For any two distinct points $\mathbf{q}_1$ and $\mathbf{q}_2$ in $\mathbb{F}_q^2$, the \textit{bisector line} $l_{\mathbf{q}_1, \mathbf{q}_2}$ is defined as
\[l_{\mathbf{q}_1, \mathbf{q}_2}:=\left\lbrace \mathbf{x}\in \mathbb{F}_q^2\colon ||\mathbf{x}-\mathbf{q}_1||=||\mathbf{x}-\mathbf{q}_2||\right\rbrace.\]

 Let $\L$ be the multi-set of bisector lines defined as \[\L=\bigcup_{(\mathbf{q_1}, \mathbf{q_2})\in \mathcal{E}^2}l_{\mathbf{q_1}, \mathbf{q_2}}.\]
 
  One can identify each point $(a,b)\in \mathcal{E}$ with a vertex $[a,b,1]$ in the Erd\H{o}s-R\'{e}nyi graph $PG(q,3)$, and each line  of the form $cx+dy+e=0$ in $\L$ with a vertex $[c,d,e]$ in the Erd\H{o}s-R\'{e}nyi graph  $\mathcal{ER}(\mathbb{F}_q^3)$. We denote the corresponding sets of vertices in the Erd\H{o}s-R\'{e}nyi graph  $\mathcal{ER}(\mathbb{F}_q^3)$ by $\E'$ and $\L'$, respectively.
  
Hence, we have $|\E'|=|\E|$, $|\L'|=|\L|$, and the sum $\sum_{\lambda\in \mathbb{F}_q^*}H_{\lambda}(\mathcal{E})$ is equal to the number of edges between $\E'$ and $\L'$ in $\mathcal{ER}(\mathbb{F}_q^3)$. Therefore, by Lemmas \ref{edge} and \ref{erdosrenyi} we get 
\[\sum_{\lambda\in \mathbb{F}_q^*}H_{\lambda}(\mathcal{E})=I(\mathcal{E}', \L')\le \frac{|\mathcal{E}||\L|}{q}+q^{1/2}|\mathcal{E}|^{1/2}\sqrt{\sum_{l\in \L}w(l)^2},\]
where $w(l)$ is the multiplicity of $l\in \L$. 

If $l_{\mathbf{x}, \mathbf{z}}=l_{\mathbf{y}, \mathbf{w}}$ and $||\mathbf{x}-\mathbf{z}||\ne 0$, then one can check that there exists a unique reflection $R_\mathbf{u}$ such that $R_\mathbf{u}(\mathbf{x})=\mathbf{z}$, $R_\mathbf{u}(\mathbf{y})=\mathbf{w}$, and $||\mathbf{x}-\mathbf{y}||=||\mathbf{z}-\mathbf{w}||$.  Thus the sum $\sum_{l\in \L}w(l)^2$ is the cardinality of the following set
\[Q:=Q(\E)=\left\lbrace (\mathbf{x}, \mathbf{y}, \mathbf{z}, \mathbf{w})\in \mathcal{E}^4\colon \exists \lambda\in \mathbb{F}_q^*,  \left((\mathbf{x},\mathbf{y}), (\mathbf{z}, \mathbf{w})\right)\in E(RF_\lambda(\mathbb{F}_q^2))\right\rbrace.\]
We set
\[Q_\lambda:=\left\lbrace (\mathbf{x}, \mathbf{y}, \mathbf{z}, \mathbf{w})\in Q\colon ||\mathbf{x}-\mathbf{y}||=||\mathbf{z}-\mathbf{w}||=\lambda\right\rbrace,\]
and see that
\[\sum_{l\in \L}w(l)^2=\sum_{\lambda\in \mathbb{F}_q^*}|Q_\lambda|.\]
For each $\lambda\ne 0$, it follows from Lemma \ref{reflection} and Lemma \ref{edge} that 
\[|Q_\lambda|\le \frac{\nu_\E(\lambda)^2}{q}+2(q-1)\nu_\E(\lambda).\]
Hence, we get
\begin{equation}\label{lund12}\sum_{l\in \L}w(l)^2\le \frac{\sum_{\lambda \in \mathbb{F}_q^*}\nu_{\mathcal{E}}(\lambda)^2}{q}+2(q-1)|\mathcal{E}|^2.\end{equation}
Combining (\ref{lund11}) with (\ref{lund12}), we obtain 
\[
\sum_{\lambda\in \mathbb{F}_q^*}\nu_{\mathcal{E}}(\lambda)^2\le |\E|\left( \frac{|\E||\L|}{q}+q^{1/2}|\E|^{1/2}\sqrt{\frac{\sum_{\lambda\in \mathbb{F}_q^*}\nu_{\mathcal{E}}(\lambda)^2}{q}+2(q-1)|\mathcal{E}|^2}\right).\]
Solving this inequality leads to the desired bound, and the proposition follows.
\end{proof}
\begin{remark}\label{0388}
It follows from the proof of Proposition \ref{pro2} that the number of the number of hinges determined by points in $\E$ is at most 
\[\frac{|\E|^3}{q}+q^{1/2}|\E|^{1/2}\sqrt{\frac{\sum_{\lambda\in \mathbb{F}_q^*}\nu_{\mathcal{E}}(\lambda)^2}{q}+2(q-1)|\mathcal{E}|^2}.\]
On the other hand, we have proved that \[\sum_{\lambda\in \mathbb{F}_q^*}\nu_{\mathcal{E}}(\lambda)^2\ll \frac{|\mathcal{E}|^4}{q}+q|\mathcal{E}|^{5/2}.\]
This leads to that the number of hinges is at most $\ll |\E|^3/q$ when $|\E|\gg q^{4/3}$.
\end{remark}

As an application of Proposition \ref{pro2}, we obtain the following result.
\begin{theorem}[Bennett et al. \cite{groupaction}]
Let $\E$ be a set of points in $\mathbb{F}_q^2$. If $|\E|\gg q^{8/5}$, then $\E$ determines a positive proportion of all triangles.
\end{theorem}
\begin{proof}
Since $|\E|\gg q^{8/5}$, without loss of generality, we assume that there are no two points $\mathbf{x}$ and $\mathbf{y}$ in $\E$ satisfying $||\mathbf{x}-\mathbf{y}||=0$. The proof of Theorem \ref{tamgiac} is very similar to that of Theorem \ref{main}, and there is the only one different step. That is, instead of using Lemma \ref{mainlm}, we use Proposition \ref{pro2}, thus we leave the rest to the reader.
\end{proof}
\section{Proof of Theorem \ref{khoangcachphanbiet}}
In order to prove Theorem \ref{khoangcachphanbiet}, we make use of the following theorem on the cardinality of a maximal independent set of a hypergraph due to Spencer \cite{spencer}.
\begin{theorem}\label{spencer}
Let $H$ be a $k$-uniform hypergraph of $n$ vertices and $m$ edges with $ m \ge n/k$, and let $\alpha(H)$ denote the independence number of $H$. Then 
$$ \alpha(H) \ge \left(1-\frac{1}{k}\right) \left \lfloor \left(\frac{n^k}{km}\right)^{\frac{1}{k-1}}\right \rfloor.$$
\end{theorem}
\begin{proof}[{\bf Proof of Theorem \ref{khoangcachphanbiet}}] 

We call a $4$-tuple of distinct elements in $\E^4$ \textit{regular} if all six generalized distances determined are distinct. Otherwise, it is called \textit{singular}. Let $H$ be the $4$-uniform hypergraph on the vertex set $V(H)=\E$, whose edges are the singular $4$-tuples of $\E$. 

On one hand, it follows from the remark (\ref{0388}) that the number of $4$-tuples containing a triple induced a hinge  is at most $((1+o(1))|\E|^3/q)\cdot |\E|=(1+o(1))|\E|^4/q$ when $|\E|\gg q^{4/3}$. Thus the number of edges of $H$ containing a triple induced a hinge is at most $(1+o(1))|\E|^4/q$

On the other hand, according to Proposition \ref{pro2} and from the fact that the number of quadruples with zero-distances is no more than $4q^4$, the number of edges of $H$ that do not contain any hinge is a most $(1+o(1))\frac{|\E|^4}{q}$, when $|\E|\gg q^{4/3}$.

In other words, if $|\E|\gg q^{4/3}$, we have  $$|E(H)|\le (1+o(1))\frac{|\E|^4}{q}.$$
By Theorem \ref{spencer}, one has $$\alpha(H)\ge C \left \lgroup\frac{|\E|^4}{|E(H)|}\right\rgroup^{1/3}= Cq^{1/3},$$
for some positive constant $C$.
Since there is no repeated distance determined by the independent set of $H$,  there exists a distinct distance subset $U\subseteq \E$ satisfying $|U|\ge \alpha(H)\ge Cq^{1/3}$.

Moreover, it is easy to see that there is at least one repeated distance determined by any set of  $\sqrt{2}q^{1/2}+1$ elements  since there are only $q=|\mathbb{F}_q|$ distances over $\mathbb{F}_q^2$. This concludes the proof of the theorem.
\end{proof}

\section{Acknowledgments}
The second listed author was partially supported by Swiss National Science Foundation grants 200020-162884 and 200020-144531.

\vspace{1cc}
\hfill\\
University of Education, \\
Vietnam National University\\
E-mail: phamduchiepk6@gmail.com
\bigskip\\
Department of Mathematics,\\
EPF Lausanne\\
Switzerland\\
E-mail: thang.pham@epfl.ch\\
\bigskip\\
University of Education, \\
Vietnam National University\\
Viet Nam\\
E-mail: vinhla@vnu.edu.vn

\end{document}